\documentclass[a4paper]{scrartcl}

\usepackage{amsfonts,amsthm,amssymb,amsmath}
\usepackage[latin1]{inputenc}
\usepackage[T1]{fontenc}
\usepackage{graphicx}
\usepackage{enumitem}
\usepackage{hyperref}
\usepackage{tikz}
\usepackage{tkz-berge}

\newtheorem{thm}{Theorem}
\newtheorem{lem}[thm]{Lemma}
\newtheorem{prp}[thm]{Proposition}

\newtheorem{question}[thm]{Question}

\newcommand{\Ord}{\operatorname{Ord}}

\title{Cops, robbers, and infinite graphs}
\author{Florian Lehner\thanks{The author acknowledges the support of the Austrian Science Fund (FWF), project W1230-N13.}}

\begin{document}
\maketitle

\begin{abstract}
Cops and robbers is a game between two players, where one tries to catch the other by moving along the edges of a graph. It is well known that on a finite graph the cop has a winning strategy if and only if the graph is constructible and that finiteness is necessary for this result.

We propose the notion of weakly cop-win graphs, a winning criterion for infinite graphs which could lead to a generalisation. In fact, we generalise one half of the result, that is, we prove that every constructible graph is weakly cop-win. We also show that a similar notion studied by Chastand et al.\ (which they also dubbed weakly cop-win) is not sufficient to generalise the above result to infinite graphs.

In the locally finite case we characterise the constructible graphs as the graphs for which the cop has a so-called protective strategy and prove that the existence of such a strategy implies constructibility even for non-locally finite graphs.
\end{abstract}

\section{Introduction}
\label{sec:intro}
The game of cops and robbers was first studied by Nowkowski and Winkler~\cite{zbMATH03801591} and Quilliot~\cite{quilliot-thesis,quilliot-habil}. Since its first appearance numerous variants and aspects of the game have been studied. The book~\cite{zbMATH05945963} by Bonato and Nowakowski gives a fairly good overview.

The game is played on the vertex set of a graph between two players called the cop and the robber. Both players have perfect information. They alternately take turns, with the cop going first. In the first round their move is selecting a starting vertex. In each consecutive round they can either move to a neighbour of the vertex they are at, or stay where they are.

The goal of the cop is to ``catch'' the robber, that is, to occupy the same vertex as the robber after finitely many steps. The robber wins if he can avoid being caught forever. Since one of the two events must happen it is clear that one of the two players has a winning strategy.

While naturally the focus has been mostly on finite graphs there also have been several publications treating the infinite case \cite{zbMATH05648096, zbMATH05826379, zbMATH01533339, zbMATH01839020, zbMATH00716426, zbMATH01417400, zbMATH01968008}.

One of the reasons why infinite graphs have not received more attention is that many very basic results break down as soon as we leave the realm of finite graphs. The most striking example of this is probably that graphs which contain an isometric copy of an infinite path can never be cop-win. Even worse, the game cannot be won by any finite number of cops on such a graph. The robber's strategy would simply consist of starting ``further out'' along the path than any cop and then running away in a straight line. Clearly the cops can never catch up and hence the robber can avoid being captured forever. 

But even if we remove this obvious obstruction, results fail to generalise. For example, Hahn et al.~\cite{zbMATH01839020} showed that there are infinite chordal graphs of diameter $2$ which are not cop-win. This contrasts the fact that finite chordal graphs are always cop-win.

In this paper we study an altered winning criterion which seems to be better adapted to infinite graphs. It coincides with the original winning criterion for finite graphs, but allows to generalise some results to infinite graphs. We only consider the most basic version of the game where one cop tries to catch one robber, but it is certainly possible that a similar approach works for more general variants of the game as well.

The main motivation for this paper is that there does not seem to be a counterpart for the following necessary and sufficient condition for the existence of a winning strategy of the cop. The condition was discovered independently by Nowkowski and Winkler~\cite{zbMATH03801591} and Quilliot~\cite{quilliot-thesis}.
\begin{thm}
\label{thm:nowakowski}
A finite graph is cop-win if and only if it is constructible.
\end{thm}

Constructible here means that $G$ can be constructed according to certain rules which will be explained in the next section. The characterisation does not remain valid for infinite graphs. One reason for this is that as mentioned before graphs with an isometric copy of an infinite path can never be cop-win.

While this problem was already addressed by Chastand~et~al.~\cite{zbMATH01533339}, their proposed solution turns out to be unsatisfactory. They introduced the notion of C-weakly cop-win graphs\footnote{Of course, Chastand~et~al.\ call them weakly cop-win. The reason we call them C-weakly cop-win is that we would like to reserve the term weakly cop-win for the new winning criterion introduced in Section \ref{sec:weaklynew}.} where the cop wins the game if he can either catch the robber or chase him away. They proved that certain infinite constructible graphs are cop-win in this sense and asked whether the C-weakly cop-win graphs are exactly the constructible graphs.

In this paper we show that this is not the case if we use their exact definition. However, taking a slightly modified definition of weakly cop-win graphs which follows the same intuition we are able to prove that the cop has a winning strategy on every constructible graph.

We also introduce protective strategies, which are a special kind of winning strategies and show that every graph which admits such a protective strategy is constructible. For locally finite graphs we can even show that being constructible is equivalent to the existence of a protective strategy.

Finally we investigate dismantable graphs (which in the finite case are exactly the constructible graphs) and give a sufficient condition for an infinite dismantable graph to be weakly cop-win.

The rest of this paper is structured as follows. After introducing some basic notions, we outline the rules of the game and give a proof of Theorem~\ref{thm:nowakowski}. We briefly discuss, why a similar characterisation is not possible for infinite cop-win graphs.  In Section~\ref{sec:weakly} we outline the approach of Chastand~et~al.\ and show that there is a locally finite constructible graph which is not C-weakly cop-win. We then proceed to introduce our modified definition of weakly cop-win graphs and prove the results mentioned above.

The main result of Section \ref{sec:construct} and probably of the whole paper is Theorem \ref{thm:weaklycopwin} which states that every constructible graph is weakly cop-win. In Section \ref{sec:protective} we introduce the protective strategies mentioned earlier and show that a locally finite graph is constructible if and only if it admits a protective strategy. Section \ref{sec:dismantable} contains some results on dismantable graphs. We conclude the paper with some interesting open questions about weakly cop-win graphs.

\section{Basic notions and auxiliary results}

All graph theoretical notions which are not explicitly defined will be taken from \cite{MR2159259}. Since throughout most of this paper the vertex sets of the graphs in consideration are well ordered, we start by recalling some facts about well orders and ordinal numbers.

It is well known that every well order is order isomorphic to some ordinal number. We denote by $\Ord$ the set of ordinal numbers. Recall that the ordinals themselves are ordered with respect to being an initial piece of one another. This order is such that we can identify $\alpha$ with $\{\beta \in \Ord \mid \beta < \alpha\}$ by means of an order isomorphism.

We denote by $\mathbb N= \{0,1,2,3,\ldots\}$ the smallest infinite ordinal (which is identical to the set of finite ordinals). Note that we do not distinguish between the set of finite ordinals and the set of non-negative integers. This can be justified by the fact that there is an order isomorphism between the two sets through which we can identify them. If $\alpha \in \Ord$ we denote by $\alpha+1 = \alpha \cup \{\alpha\}$ its successor. For a successor ordinal $\alpha$ we denote by $\alpha -1$ its predecessor.  If an order is order isomorphic to $\alpha \leq \mathbb N$ we call it a \emph{natural order}.

Let $G=(V,E)$ be a graph and let $V$ be well ordered. Then there is an order isomorphism $\phi \colon V \to \alpha$ for some $\alpha \in \Ord$. Hence we can without loss of generality assume that $V = \alpha = \{\nu \in \Ord \mid \nu < \alpha\}$ and that the well order on $V$ is the same as the order of $\alpha$. In particular, the smallest vertex will always be $0$. We denote by $G_{<\nu}$ the subgraph of $G$ induced by all $\gamma < \nu$ and by  $G_{\leq\nu}$ the subgraph of $G$ induced by all $\gamma \leq \nu$. For example, $G_{<0}$ is the empty graph, $G_{<1}$ consists only of the vertex $0$, and $G_{<V} =G$.

All graphs considered in this paper are undirected and reflexive, that is, there is at least one loop attached to each vertex. This is done mostly because it is convenient for defining the rules of the cops-and-robbers-game. Staying at a vertex simply amounts to taking a step along a loop in this vertex.

Let $G,G'$ be graphs. A \emph{homomorphism} from $G$ to $G'$ is an adjacency preserving map between the vertex sets. A \emph{retraction} is a graph homomorphism from $G$ onto a subgraph $H$ of $G$ whose restriction on $H$ is the identity. We say that $H$ is a \emph{retract} of $G$ if there is a retraction $\phi\colon G \to H$. Note that since all graphs in this paper are reflexive it is possible that adjacent vertices are mapped to the same vertex. In particular, every subgraph consisting of a single vertex with a loop attached is a retract of $G$.

Let $G=(V,E)$ be a graph and let $\nu,\mu \in V$. We say that $\mu \neq \nu$ \emph{dominates} $\nu$ if $\mu$ is adjacent to $\nu$ and all of its neighbours. 

A graph $G$ is called \emph{constructible} if there is a well order $<$ of its vertex set $V$, such that each vertex $\nu$ is dominated in $G_{\leq \nu}$. In other words, $G$ can be constructed from a single vertex by recursively adding dominated vertices. An order $<$ with the aforementioned properties is called a \emph{dominating order}. A map $\delta$ which maps every vertex $\nu$ to a vertex which dominates $\nu$ in $G_{\leq \nu} $ is called a \emph{domination map} associated to the dominating order $<$.

The following lemma will be used later.

\begin{lem}
\label{lem:finitedominationchain}
Let $G=(V,E)$ be a constructible graph with dominating order $<$ and an associated domination map $\delta$. For every $\nu \in V$  there is $k \in \mathbb N$ such that $\delta^k(\nu) = 0$.
\end{lem}

\begin{proof}
The statement follows from the facts that $\delta(\nu) < \nu$ for every $\nu$ and that there are no infinite descending chains in a well order. Hence after finitely many iterations we must arrive at the minimal element $0$.
\end{proof}

Next we would like to construct a retraction from $G$ onto $G_{<\nu}$. Define a map $\rho_\nu$ by
\[
	\rho_\nu(\mu) = \delta^{k(\nu,\mu)}(\mu),
\]
where
\[
k(\nu,\mu) = \min\{k \in \mathbb N \mid \delta^k(\mu) < \nu\}.
\] 
Note that by the above lemma this value is well defined and finite for every $\nu \geq 1$ because in this case $0 < \nu$.

\begin{lem}
\label{lem:retraction}
The map $\rho_\nu$ is a retraction from $G$ onto $G_{<\nu}$.
\end{lem}

\begin{proof}
First observe that $k(\nu,\mu) =0$ if  $\mu < \nu$. So in this case $\rho_\nu(\mu) = \mu$ and thus the restriction of $\rho_\nu$ to $G_{<\nu}$ is indeed the identity map.

It remains to show that $\rho$ is a homomorphism, that is, adjacent vertices map to adjacent vertices or to the same vertex. This is done by transfinite induction. More precisely, for every $\gamma \in \Ord$ such that $\nu \leq \gamma \leq V$ we show that the restriction of $\rho_\nu$ to $G_{<\gamma}$ is a homomorphism from $G_{<\gamma}$ to $G_{<\nu}$.

For $\gamma = \nu$ we get the identity map which clearly is a homomorphism. Now assume that the statement was true for every $\gamma' < \gamma$. If $\gamma$ is a limit ordinal, then each edge of $G_{<\gamma}$ appears in some $G_{<\gamma'}$ for $\gamma' < \gamma$ and hence is mapped to an edge of $G_{<\nu}$. 

On the other hand, if $\gamma$ is a successor ordinal, then $G_{<\gamma}$ is obtained from the predecessor $G_{<\gamma-1}$ by adding the vertex $\gamma-1$. Furthermore there is a vertex $\mu = \delta(\gamma-1)$ which dominates $\gamma-1$ in $G_{<\gamma}$. By definition of $\rho_\nu$ it follows that $\rho_\nu(\gamma-1) = \rho_\nu(\mu)$. Since every edge incident to $\mu$ is mapped to an edge of $G_{<\nu}$ and the neighbours of $\gamma-1$ are a subset of the neighbours of $\mu$ we conclude that every edge incident to $\gamma-1$ is also mapped to an edge of $G_{<\nu}$. Since these are the only edges of $G_{<\gamma}$ which were not present in $G_{<\gamma-1}$ this completes the induction step and thus also the proof of the lemma.
\end{proof}

Sometimes it is beneficial to have a natural order rather than an arbitrary well order. The following lemma shows, that in the case of locally finite graphs we always can restrict ourselves to natural dominating orders.

\begin{lem}
\label{lem:naturalorder}
A locally finite graph is constructible if and only if it admits a natural dominating order.
\end{lem}

\begin{proof}
The ``if''-part is trivial, hence we only need to show that every constructible graph admits a natural dominating order. 

For this purpose let $<$ be an arbitrary dominating order. We assign a number $n(\nu)$ to each vertex $\nu$ by transfinite recursion: $n(0)=0$, and
\[n(\nu) = \max_{\substack{\mu \in N(\nu)\\\mu<\nu}} n(\mu) +1\] 
for all other vertices, where $N(\nu)$ denotes (as usual) the neighbourhood of the vertex $\nu$ in $G$. Note that this maximum exists and is an integer since every vertex has only finitely many neighbours.

Order the vertices by their values of $n(\nu)$ where vertices with the same value are put in arbitrary order and denote the resulting order by $\prec$. We claim that $\prec$ is the desired dominating order.

It is easy to see that $\prec$ is order isomorphic to $\mathbb N$. Simply observe that $n(\nu)$ is at least as big as the distance between $\nu$ and $0$. Hence for each value $k$ there are only finitely many vertices with $n(\nu) = k$ and thus only finitely many vertices can appear before a given vertex in the order.

Finally observe that the neighbours of any vertex $\nu$ in $G_{\preceq \nu}$ are exactly the same as in $G_{\leq \nu}$ because every neighbour $\mu$ of $\nu$ with $\mu>\nu$ has $n(\mu) \geq n(\nu)+1$. This implies that $\nu$ is dominated in $G_{\preceq \nu}$ by the same vertex as in $G_{\leq \nu}$. Since this is true for every vertex we have established that $\prec$ is dominating.
\end{proof}

\section{The rules of the game}
\label{sec:notations}

The game of cops and robbers is a perfect information game played by 2 players---the cop and the robber. Following the notation of Bonato and Nowakowski~\cite{zbMATH05945963}, throughout this paper the cop will be female while the robber will be male. The players take turns in discrete time steps or \emph{rounds} indexed by the non-negative integers. The cop plays in rounds with an even index while the robber plays in rounds with an odd index. The rules of the game are as follows. 

First both players must choose their respective starting points, the cop in round 0 and the robber in round 1. In each subsequent round the players can move along one edge. It should be clear from the context what we mean when we say the cop is at (or occupies) a vertex $v \in V$ after round $n \in \mathbb N$ and the like. 

The cop wins the game, if after some finite number of rounds both players occupy the same vertex, otherwise the robber wins. 

Since one or the other must happen it follows from von Neumann's theorem that one of the two players has a winning strategy. A \emph{strategy} of the cop in this context is a function $S\colon V \times V \to V$ which assigns to $(c,r)$ a neighbour of $c$. If it is the cop's turn and the cop occupies $c$ and the robber occupies $r$ then the cop moves to $S(c,r)$.  A strategy of the cop is called \emph{winning} if following this strategy will always lead to a win of the cop, no matter what the robber does. Call a graph $G$ \emph{cop-win} if the cop has a winning strategy on $G$. Otherwise call it \emph{robber-win}.

 It is an easy observation that if $c$ dominates $r$, then the cop can win the game on her next turn whenever she occupies $c$ and the robber occupies $r$. In the light of this it is not surprising that constructible graphs are cop-win. The following classical result which was already mentioned in Section~\ref{sec:intro} was discovered independently by Nowakowski and Winkler~\cite{zbMATH03801591} and Quilliot~\cite{quilliot-thesis}. We include a short proof, not only for the convenience of the reader, but also since we will reference some of the proof ideas later in this paper.

\begingroup
\def\thethm{\ref{thm:nowakowski}}
\begin{thm}
A graph is cop-win if and only if it is constructible.
\end{thm}
\addtocounter{thm}{-1}
\endgroup

\begin{proof}
We use induction on the number of vertices. Clearly the statement is true for the graph on one vertex which is both constructible and cop-win. Now assume that the equivalence holds for every graph on at most $n-1$ vertices and let $G =(V,E)$ be a graph on $n$ vertices.

If $G$ is cop-win, then there must be vertices $c$ and $r$ occupied by the cop and the robber before the last move of the robber. We assume that the robber plays optimally, that is, he tries to avoid capture as long as possible. Then $r$ is dominated by $c$ because otherwise the robber could avoid being captured for at least one more round. 

It is easy to see that $G - r$ is cop-win. Simply note that the cop can play the exact same strategy as on $G$ with the only difference that he uses $c$ instead of $r$.

By induction hypothesis $G-r$ is constructible. By appending $r$ to any dominating order of $G-r$ we obtain a dominating order of $G$ because $r$ is dominated by $c$.

Conversely assume that $G$ is constructible. Let $v$ be the last vertex in a dominating order of $G$. Then $G-v$ is constructible and hence cop-win. Furthermore there is a vertex $u \in V$ which dominates $v$. 

Assume that the cop plays a winning strategy for $G-v$ on $G$, pretending that the robber is at $u$ whenever he moves to $v$. Then after finitely many steps the cop either catches the robber, or he catches the robbers ``shadow'', meaning that the cop is at $u$ while the robber is at $v$. In the latter case the cop will win the game in the next move, proving that $G$ is cop-win.
\end{proof}

Usually, in the above theorem the notion dismantability is used in place of constructibility. A graph is dismantable if we can recursively remove dominated vertices until we end up with a single vertex, the order in which the vertices are removed is called a dismantling order. Clearly, dismantability and constructibility are equivalent conditions for finite graphs as the reverse of every constructing order is a dismantling order and vice versa. 

This equivalence is not valid for infinite graphs. The reason for this is that both constructing and dismantling orders are required to be well orders, but reversing a well order does not give a well order again. In this paper we use constructibility rather than dismantability because it seems to be better adapted to the inductive approach that we take. Nevertheless, in Section \ref{sec:dismantable} we will consider dismantable graphs as well.

A different notion of strategy which will be used later in this paper are so called \emph{time dependent strategies}. As the name suggests, in this kind of strategy the cop's move does not only depend on the positions of the two players, but also on the number of the current round. In other words a time dependent strategy is a map $s\colon V \times V \times \mathbb N$ which assigns to the triple $(c, r,t)$ a neighbour of $c$ meaning that if at time $t$ the cop is at $c$ and the robber is at $r$, then the cop moves to $s(c,r,t)$ in the next step.

While at the first glance this looks like a stronger concept than the ordinary strategies defined earlier, it turns out that at least in the finite case it has no impact on whether or not the cop has a winning strategy.

\begin{prp}
Let $G$ be finite. The cop has a winning strategy if and only if she has a time dependent winning strategy.
\end{prp}

\begin{proof}
Since every strategy is also a time dependent strategy one of the two implications is trivial. Now assume that $c$ has a time dependent winning strategy. Then by the same arguments as in the proof of Theorem \ref{thm:nowakowski} we can conclude that $G$ is constructible and hence there must be a winning strategy for $c$.
\end{proof}

We conjecture an analogous statement to be true in the infinite case with the notion of weakly winning which is introduced in Section \ref{sec:weaklynew} of this paper.

\section{Weakly cop-win graphs}
\label{sec:weakly}

As mentioned earlier, Theorem~\ref{thm:nowakowski} does not remain valid if we consider infinite graphs.  The issue was first addressed by Chastand et al.~\cite{zbMATH01533339}. They suggested the following change to the rules: the cop wins, if in order to avoid being caught, after some finite number of steps the robber has to move to a vertex that he has not visited so far on each move. The intuition behind this is, that the robber can only avoid being caught by running away in a straight line from some point on. 

Note that there is no a-priori bound on the number of steps after which the robber has to run away in a straight line. In fact it is easy to show that there cannot be such an a priori bound if the graph has infinite diameter. In this case, for any $n \in \mathbb N$ the robber could begin $n$ steps away from the cop and only start running away once the cop approaches him. This will not happen before she has taken at least $n-1$ turns.

We call graphs for which the cop has a winning strategy with respect to this modified winning criterion \emph{C-weakly cop-win}. Clearly countable trees and locally finite chordal graphs are C-weakly cop-win. Furthermore one can show the following.

\begin{thm}[Chastand et al.~\cite{zbMATH01533339}]
\label{thm:chastand}
Let $G$ be a constructible graph.  Assume that $G$ admits a dominating order $<$ such that there is a domination map associated to $<$ which is a self contraction of $G$. Then $G$ is C-weakly cop-win.
\end{thm}

A self contraction here is a function from $G$ to itself which maps adjacent vertices either to adjacent vertices or to the same vertex. From the above theorem it follows that every Helly graph and every connected bridged graph is C-weakly cop-win. In the light of this result it is natural to ask if every constructible graph admits a dominating order as in the condition of Theorem~\ref{thm:chastand} and whether or not the constructible graphs are exactly the C-weakly cop-win graphs (Questions 3 -- 5 in \cite{zbMATH01533339}). This is however not the case as the following example shows.

Let $G$ be a graph constructed as follows. Start with the $5$-regular tree $T_5$. Then replace every vertex $v$ of this tree by a copy the graph $H$ shown in Figure~\ref{fig:doublewheel} attaching the edges incident to $v$ to the copies of the vertices $a_0, \ldots a_4$. We would like to remark that the graph in Figure~\ref{fig:doublewheel} has been discovered independently by Boyer~et~al.~\cite{zbMATH06180528} and used to show that the distance between the cop and the robber in an optimal (with respect to the capture time) winning strategy of the cop is not necessarily monotonically decreasing. In fact, we also exploit that in order to catch the robber on this graph, the cop has to increase the distance to the robber giving him the opportunity to stay at the same vertex for another round.

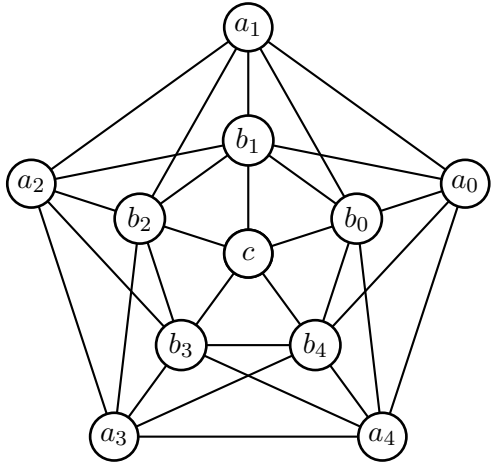
\begin{figure}
\begin{center}
\begin{tikzpicture}
	\SetVertexMath
	\SetVertexNormal[LineWidth=1pt]
	\grCycle[prefix=a,rotation=18,RA=3]{5}
	\grCycle[prefix=b,rotation=18,RA=1.5]{5}
	\EdgeMod*{a}{b}{5}{-1}{1}
	\EdgeMod*{a}{b}{5}{0}{1}
	\EdgeMod*{a}{b}{5}{1}{1}
	\Vertex{c0}
	\Vertex{c}
	\EdgeFromOneToAll{c}{b}{0}{5}
\end{tikzpicture}
\end{center}
\caption{The building block $H$ of the counterexample.}
\label{fig:doublewheel}
\end{figure}

\begin{thm}
The graph $G$ defined above is constructible, but not C-weakly cop-win.
\end{thm}

\begin{proof}
To see that $G$ is constructible pick an arbitrary root $r$ of $T_5$. Clearly, the order in which vertices of $T_5$ are visited by a breadth first search starting at $r$ is a dominating order.

Assume without loss of generality that all copies of $H$ are inserted in a way that $a_0$ is the vertex closest to the copy of $H$ that was used to replace the root. Observe that $a_0,b_4,c,b_0,\ldots,b_3,a_1,\ldots, a_4$ is a dominating order of $H$. 

Put the copies of $H$ in the same order as the corresponding vertices appear in the breadth first search and order the vertices of each copy of $H$ as above. Every copy $a$ of $a_0$ is dominated in $G_{\leq a}$ by its only neighbour outside of the corresponding copy of $H$. On the other hand, if $v$ is a copy of some other vertex of $H$, then the only neighbours of $v$ in $G_{\leq v}$ lie in the corresponding copy of $H$. Since the order of the vertices of each copy of $H$ is dominating, there must be a vertex which dominates $v$ in $G_{\leq v}$.

Hence it only remains to show that $G$ is not C-weakly cop-win. To see this, first notice that the robber can run away forever without ever visiting a copy of $b_i$ or $c$. If the cop only uses vertices of type $a_i$ and $b_i$ from some finite time on, then the robber can avoid being caught forever while staying in one copy of $H$. The only thing he has to do is move around the outer cycle to stay as far away from the cop as possible. 

Thus the cop cannot win, unless she visits infinitely many copies of the center $c$. But whenever she does this, she is at least two steps away from the robber. So he can stay at the same vertex for one step or move back to where he came from.
\end{proof}

\section{A different approach}
\label{sec:weaklynew}

In the last section we have seen that it is not possible to generalise Theorem~\ref{thm:nowakowski} to infinite graphs using the notions of Chastand~et~al. In this section we suggest another definition of weakly cop-win graphs for which it is possible to generalise at least one of the two implications of the theorem. The intuition behind this definition is quite similar to the intuition behind the definition in the last section: The cop wins by either catching the robber, or by chasing him away. However, our notion of chasing away seems to be better adapted to the game and the structure of infinite graphs.

We call a strategy of the cop \emph{weakly winning} if it prevents the robber from visiting any vertex infinitely often. In particular, if the cop wins then the robber must eventually leave every finite set of vertices. Note that if the graph $G$ is locally finite, then the robber's position must be converging to an end of $G$. Since the cop's position must converge to the same end, we can think of this as ``catching the robber at an end''. Thus the cop wins the game, if she catches the robber at either a vertex or an end of $G$. A graph $G$ is called \emph{weakly cop-win} if it admits a weakly winning strategy.

The notion of weakly cop-win graphs can be used to generalise results about finite cop-win graphs to an infinite setting--- sometimes even with the same proof that has been used for years in the finite case. The following lemma is a nice example of such a result.

\begin{lem}
\label{lem:retractcopwin}
Retracts of weakly cop-win graphs are weakly cop-win.
\end{lem}

\begin{proof}
Let $G$ be a weakly cop-win graph and let $\rho \colon G \to H$ be a retraction of $G$ to $H$. Now imagine two games played in parallel on $G$ and $H$ where the robber is not allowed to leave the subgraph $H$. 

The cop $C_G$ on $G$ plays the winning strategy while the cop $C_H$ on $H$ just shadows the moves of $C_G$ with respect to the retraction $\rho$. That is, when $C_G$ is at $\nu$, then $C_H$ will be at $\rho(\nu)$.

It is clear that if the robber is caught by either of the cops, then the other cop will catch him on the same move. Hence if the strategy of $C_H$ is not winning, then the robber is able to visit the same vertex infinitely often without being caught by either of the cops. This contradicts the assumption that the strategy on $G$ was winning.
\end{proof}

In the following theorem we restrict ourselves to locally finite graphs. This is done mainly to demonstrate how similar to the finite case the proof is. A more general version be proved in the next section. However, the proof there is quite different.

\begin{thm}
\label{thm:locfincopwin}
If a locally finite graph $G$ is constructible, then $G$ is weakly cop-win.
\end{thm}

\begin{proof}
First note that by Lemma \ref{lem:naturalorder} we can assume that the dominating order on the vertex set is a natural order. In particular $G_{<\nu}$ is finite for every $\nu \in V$. Construct the strategy $s_\nu$ of the cop for each $G_{<\nu}$ exactly like in the proof of Theorem \ref{thm:nowakowski}. 

Note that for $\mu < \nu$ the strategies $s_\mu$ and $s_\nu$ coincide on $G_{<\mu}$. This observation is crucial because it allows us to define a limit strategy $s$ on $G$ by $s(c,r) = s_\nu(c,r)$ where $\nu > \max \{c,r\}$. Clearly, with this strategy the cops position after her move is always smaller or equal to the robbers position because the same holds for each strategy $s_\nu$.

Now assume that the robber can visit some vertex $r$ infinitely many times. Since $G_{<r}$ is finite this implies that there is some vertex $c \leq r$ such that it happens more than once, that the cop is at $c$ while the robber is at $r$.

But then the robber could repeat the steps between two occurrences of this situation arbitrarily. Since this leads to the same situations over and over again, neither the cop nor the robber will visit but finitely many vertices. Hence the robber would have a winning strategy in $G_{<\nu}$ where $\nu \in V$ is a common upper bound for all the vertices visited by the cop and the robber. 

But this contradicts the fact that by Theorem~\ref{thm:nowakowski} (or rather by its proof) the cop's strategy on $G_{<\nu}$ is winning.
\end{proof}

\section{Constructible graphs, retractions, and another strategy}
\label{sec:construct}

Now let us consider a different strategy $s^*$ which is defined using the retraction maps $(\rho_\nu)_{\nu \in \Ord}$ from Lemma~\ref{lem:retraction}. Since $\rho_\nu$ is a retraction it is clear that if $uv \in E$ then $\rho_\nu(u)\rho_\nu(v) \in E$. For a fixed vertex $v$ it follows from Lemma \ref{lem:finitedominationchain} that $\rho_\nu(v)$ only takes on finitely many different values $v,\delta(v),\delta^2(v),\ldots,\delta^k(v)$.

Now the strategy for the cop is to start at $0$ (which coincides with $\rho_1(v)$ for every $v$). For each consecutive step define
\[
	s^*(c,r) = \delta^{k(c,r)}(r)
\]
where
\[
	k(c,r) = \min \{k \in \mathbb N \mid c \delta^k(r) \in E\}.
\]
In other words the cop goes through the sequence $r,\delta(r),\delta^2(r),\ldots$ and moves to the first neighbour of $c$ that she encounters.

To see that this is indeed a viable strategy recall that the cop starts at $0 = \rho_1(r)$. Now assume that the robber moves from $r^{\text{old}}$ to $r^{\text{new}}$. If the cop is at $\rho_\nu(r^{\text{old}})$ for some $\nu \in \Ord$ then she can always move to $\rho_\nu(r^{\text{new}})$ because $\rho_\nu$ is a retraction. This implies that the above definition always yields a legal move.

\begin{thm}
\label{thm:weaklycopwin}
Every constructible graph is weakly cop-win.
\end{thm}

\begin{proof}
We will show that $s^*$ is a winning strategy by bounding the number of visits of the robber to any particular vertex $\mu$.

 By definition of the strategy, after the cop's move she is always at $\rho_\nu(r)$ where $r$ is the position of the robber.  Furthermore the cop can always follow the image of the robber under $\rho_\nu$ because $\rho_\nu$ is a retraction. This implies that the $\nu$ for which $c = \rho_\nu(r)$ is non-decreasing.

Now assume that in round $n-1$ the robber moved from $r^{\text{old}}$ to $r^{\text{new}}$ and that in round $n$ the cop moved from $c^{\text{old}}$ to $c^{\text{new}}$. There are integers $k$ and $l$ such that $c^{\text{old}} = \delta^k(r^{\text{old}})$ and $c^{\text{new}} = \delta^l(r^{\text{new}})$. Let $\xi=\delta^{k-1}(r^{\text{old}})$ and $\eta = \delta^{l-1}(r^{\text{new}})$.

If $\eta < \xi$ then $G_{<\xi}$ contains $\eta$ but not $\xi$. Since $r^{\text{old}}r^{\text{new}}$ is an edge and $\rho_\xi$ is a retraction this would imply that there is an edge connecting $c^{\text{old}} = \rho_\xi(r^{\text{old}})$ and $\eta = \rho_\xi(r^{\text{new}})$. But then the cop would have moved to $\eta$ rather than to $c^{\text{new}}$. Furthermore, it is impossible that $\xi = \eta$ because in this case the cop would have moved to $\xi$ rather than staying at $c^{\text{old}}=c^{\text{new}} = \delta(\xi)$.

Hence $\xi < \eta$  and thus $G_{<\eta}$ contains $\xi$ but not $\eta$. This implies that $c^{\text{new}} = \rho_\eta(r^{\text{new}})$. Now recall that the $\nu$ for which $c = \rho_\nu(r)$ is non-decreasing. Hence if the robber visits $r^{\text{old}}$ again, then the cop will be able to move to $\xi$ and hence she will move to a vertex $c = \delta^{k'}(r^{\text{old}})$ for $k'<k$. In particular, the robber can visit the vertex $r^{\text{old}}$ at most $k-1$ more times without being caught.
\end{proof}

Of course, the above proof also works for finite graphs. One may be tempted to think, that this gives a different winning strategy for the cop than the proof of Theorem \ref{thm:nowakowski}. While technically this is the case, the difference is not as big as one may think at the first glance. To prove our point, we define a second strategy for the cop on a constructible graph which is very similar to the strategy in the aforementioned proof. 

It is defined by transfinite recursion on each graph $G_{<\nu}$. For $G_{<1}$ which only consists of the vertex $0$ there is only one possible strategy $s_1$. Now assume that we have defined strategies $s_\gamma$ on $G_{<\gamma}$ for each $\gamma < \nu$ such that for $\gamma'<\gamma<\nu$ the strategies $s_\gamma$ and$s_{\gamma'}$ coincide on $G_{<\gamma'}$.

If $\nu$ is a limit ordinal then define $s_\nu(c,r) = s_\gamma(c,r)$ where $\gamma < \nu$ is chosen big enough that both $c$ and $r$ are contained in $G_{<\gamma}$. To see that this is well defined we can use the same line of argument as in the definition of the limit strategy in Theorem \ref{thm:locfincopwin}, that is, all strategies coincide on smaller retracts.

When $\nu$ is a successor ordinal then $G_{<\nu}$ is obtained from the predecessor $G_{<\nu-1}$ by adding the vertex $\nu-1$.  There is a vertex $\mu = \delta(\nu-1)$ which dominates $\nu-1$ in $G_{<\nu}$. We now define
\[
	s_\beta(c,r)= 
	\begin{cases}
		\nu-1 &\text{if }r = \nu-1 \text{ and } cr \in E,\\
		s_{\beta-1}(c,\mu) &\text{if }r = \nu-1 \text{ and } cr \notin E,\\
		s_{\beta-1}(c,r) &\text{if }r \neq \nu-1.
	\end{cases}
\]
This defines a strategy on $G_{<\nu}$.  We denote the strategy $s_V$ on $G_{<V}=G$ obtained by the above construction by $s$. Note that just like the strategy $s^*$ this strategy is not defined on all possible configurations $(c,r)$. However, one implication of the following result is that the strategy $s$ is defined on all relevant configurations. The other implication is quite obvious: the strategy $s^*$ is in fact very similar to the strategy used for the proof in the finite case.

\begin{prp}
The strategies $s$ and $s^*$ coincide.
\end{prp}

\begin{proof}

Consider the following setting: we have two cops $C_\nu$ and $C_\nu^*$ playing the strategies $s$ and $s^*$ respectively. However, they do not play against the robber. Instead their opponent is the shadow $R_\nu$ of the robber. Whenever the robber is at some vertex $r$ then the shadow will be at the vertex $r^\nu = \rho_\nu(r)$. The game ends as soon as either of the cops has caught $R_\nu$. Since in both strategies the cop's position after her move is always smaller or equal to the robber's position it is immediate, that the game is really played on $G_{<\nu}$. This observation is crucial since it allows an inductive approach.

We claim that for any given trajectory of the robber, the corresponding trajectories of $C_\nu$ and $C_\nu^*$ coincide. This is sufficient to prove the proposition because for every position $(c,r)$ which can occur in the game there is a finite number of steps leading to this position. Hence there is some $\nu \in \Ord$ such that the robber's trajectory until reaching this position is contained in $G_{<\nu}$ and thus coincides with the trajectory of $R_\nu$. In other words, $C_\nu$ and $C_\nu^*$ have been playing against the real robber and hence
\[
s(c,r) = s(c,r^\nu) = s^*(c,r^\nu) = s^*(c,r).
\]

It only remains to prove the claim. As mentioned earlier we will use induction on $\nu$. Furthermore, for each $\nu$ we will use induction on the number of steps taken by the robber.

Clearly, for $\nu = 1$ the claim is satisfied as $G_{<1}$ consists only of the vertex $0$ and thus both of the cops win against the $R_\nu$ in the first step. In particular we do not need the induction on the robber's steps. For every $\nu > 1$ the first steps of $C_\nu$ and $C_\nu^*$ coincide as well because both strategies use $0$ as a starting vertex.

For the induction step assume that the statement is true for every $\gamma < \nu$ and that the trajectories of $C_\nu$ and $C_\nu^*$ coincided until reaching some position $(c,r)$. If there is some $\gamma < \nu$ such that $r^\gamma = r^\nu$, then by the induction hypothesis
\begin{align*}
	s^*(c,r^{\nu})= s^*(c,r^{\gamma}) = s(c,r^{\gamma}) = s(c,r^{\nu})
\end{align*}
and we are done. This is in particular the case if $\nu$ is a limit ordinal.

Now assume that this was not the case. Then $\nu$ is a successor ordinal and $r^{\nu} = \nu-1$. If $c$ is incident to $\nu-1$ then both cops move straight to $\nu-1$. Otherwise, in both strategies the cop will do the exact same thing as if the robber was at $\delta(\nu-1)$. Since $\delta(\nu-1) < \nu-1$ we can use the induction hypothesis to show that the two strategies again must coincide.
\end{proof}

\section{Protective strategies}
\label{sec:protective}

In this section we would like to present a result towards the converse implication. To formulate this result we need a special kind of strategy for the cop, which will be called a \emph{protective strategy}. 

We say that the robber \emph{robs a vertex $\nu$ at time $n$} if he is at $\nu$ after round $n$ but does not get caught in round $n+1$. Let $s$ be a (possibly time dependent) strategy of the cop. For every vertex $\nu$ define
\[
t_r^s (\nu) = \sup \{t \in \mathbb N \mid \text{the robber robs }\nu \text{ at time }t\}
\]
where the supremum is taken over $t$ and over all possible strategies of the robber. Furthermore define
\[
t_c^s (\nu) = \inf \{t \in \mathbb N \mid \text{the cop is at }\nu \text{ after round } t\}
\]
where again the infimum is taken over $t$ and over all strategies of the robber. Call a strategy protective, if
\[
\forall \nu \in V\colon  t_r^s (\nu) <  t_c^s (\nu) < \infty.
\]
The reason why we call this a protective strategy is, that the cop protects an increasing proportion of the graph in the following sense. As soon as there is a possibility for the cop to visit a vertex according to the strategy, then the vertex cannot be robbed any more since an attempt to do so will lead to an immediate capture of the robber.

Note that a protective strategy is always winning, because $t_c^s (v)$ gives an a priori bound on how many times the robber can visit a certain vertex.

\begin{thm}
\label{thm:protective}
Let $G$ be a graph on which there exists a time dependent protective strategy for the cop. Then $G$ is constructible.
\end{thm}

\begin{proof}
We order the vertices by the values of $t_r^s (\nu)$. Vertices with equal values are put in an arbitrary order such that the resulting order $<$ is a well order. It is easy to see (e.g.\ using transfinite induction) that this is always possible.

We claim that $<$ is a dominating order. Let $\nu \in V$ and assume that the robber plays a strategy which allows him to rob $\nu$ at time $t_r^s (\nu)$. 

Clearly, the cop must be at a neighbour $\mu$ of $\nu$ because otherwise the robber could stay at $\nu$ for another round which contradicts the definition of $t_r^s (\nu)$. Furthermore, if the robber moves to some vertex in $G_{\leq \nu}$ he will be caught immediately because $t_r^s (\eta) \leq t_r^s (\nu)$ for every $\eta<\nu$. This implies that $\mu$ is adjacent to every neighbour of $\nu$ in $G_{\leq \nu}$.

Finally we need to show that $\mu \in G_{\leq \nu}$. Since the cop is at $\mu$ at time $t_r^s (\nu)$ we have $t_c^s (\mu) \leq t_r^s (\nu)$ and because $s$ is protective $t_r^s (\mu) < t_c^s (\mu)$ which implies that $\mu<\nu$ in the order.
\end{proof}

\begin{thm}
\label{thm:constructible}
A locally finite graph $G$ is constructible if and only if  $G$ admits a time dependent protective strategy.
\end{thm}

\begin{proof}
One of the two implications is a direct consequence of Theorem \ref{thm:protective}.

For the converse implication assume that $G$ is constructible and let $<$ be a dominating order of the vertices. By Lemma \ref{lem:naturalorder} we can assume that $<$ is is order isomorphic to $\mathbb N$.

The strategy we would like to play is very similar to the strategy $s^*$ in the proof of Theorem~\ref{thm:weaklycopwin}. The main difference is that we need to make sure that the cop does not move to a vertex too early (earlier than $t_r^s (\nu)$). For this purpose let $\delta$ be a domination map associated with the constructing order $<$ and let 
\[
\rho_\nu \colon G \to G_{< \nu}
\]
be the retraction from Lemma \ref{lem:retraction}. Now the strategy of $c$ looks as follows:
\[
	s(c,r,2\nu) = \rho_{\nu+1}(r).
\]
The $2\nu$ in the above formula simply accounts for the fact that the cop only gets to move every second round. 

If it is possible to play this strategy then it is easy to see that it is protective. Simply observe that $c$ moves to a vertex $\nu$ no earlier than round $2\nu$, thus $t_c^s (\nu) \geq 2\nu$. In fact it holds that $t_c^s (\nu) = 2\nu$ since $\rho_{\nu+1}(\nu) = \nu$ and thus a robber who just stays at $\nu$ will be caught after $2\nu$ steps. On the other hand, if the robber moves to $\nu$ at some time $t>2\nu-1$, then he will be caught immediately because in this case $\rho_{\frac{t+1}{2}}(\nu) = \nu$. Hence he cannot rob $\nu$ at any time $t>2\nu-1$.

It remains to show that the strategy defined above is indeed a viable strategy for the cop, that is, the cop can always move to $\rho_{\nu+1}(r)$. For this purpose it suffices to show that whenever $\mu$ and $\eta$ are adjacent, there is also an edge connecting $\rho_{\nu}(\mu)$ and $\rho_{\nu+1}(\eta)$. Notice that there is an edge (possibly a loop) connecting $\rho_{\nu+1}(\mu)$ and $\rho_{\nu+1}(\eta)$ because $\rho_{\nu+1}$ is a retraction by Lemma \ref{lem:retraction}.

If $\rho_{\nu+1}(\mu) \neq \nu$, then $\rho_\nu(\mu) = \rho_{\nu+1}(\mu)$ and we are done. If $\rho_{\nu+1}(\mu) = \nu$, then $\rho_{\nu}(\mu)$ dominates $\rho_{\nu+1}(\mu)$ in $G_{\leq \nu}$. Thus every neighbour of $\rho_{\nu+1}(\mu)$ (in particular $\rho_{\nu+1}(\eta)$) is also a neighbour of  $\rho_\nu(\mu)$.

This completes the proof that constructibility implies the existence of a time dependent protective strategy and thus also the proof of the theorem.
\end{proof}

\section{Dismantable graphs}
\label{sec:dismantable}

 We call a graph \emph{dismantable} if there is a well order $<$ of the vertex set such that every vertex $\nu$ is dominated by some vertex $\mu$ in $G_{\geq \nu}$ where $G_{\geq\nu}$ is the subgraph of $G$ induced by all vertices that are $\geq \nu$. The order $<$ is called a  \emph{dismantling order}. To this dismantling order we can associate a \emph{domination map} which maps each vertex $\nu$ to some vertex which dominates it in $G_{\geq \nu}$. 
 
Dismantability and constructibility are clearly equivalent for finite graphs, but they are quite different concepts for infinite graphs. The reason for this is, that reversing a well order does not give a well order again unless the ordered set is finite.
 
It is relatively easy to see that there are infinite weakly cop-win graphs which are not dismantable. Simply consider  any leafless tree. Since every vertex has at least two neighbours and those neighbours must be non-adjacent by acylicity of the tree there cannot be a dominated vertex in such a graph. Hence the tree is not dismantable. However, it is easy to see that any tree is constructible and thus weakly cop-win.

Conversely one can show that there are dismantable graphs which are not weakly cop-win. The following example is from \cite{zbMATH01533339}. Take an infinite path $a_0, a_1, a_2, \ldots$ and a finite path $b_0,b_1,b_2$ and connect $b_0$ and $b_2$ by an edge to every $a_i$. Denote the resulting graph by $G$.  It is easy to see that $a_0, a_1, \ldots,b_0,b_1,b_2$ is a dismantling order. 

To see that $G$ is not weakly cop-win consider the map $\rho$ defined by $\rho(a_i) = a_0$ for every $i \in \mathbb N$ and $\rho(b_i)=b_i$ for $0 \leq i \leq 2$. This map is a retraction from the graph $G$ onto a $4$-cycle. Since this cycle is not weakly cop-win we conclude by Lemma \ref{lem:retractcopwin} that $G$ cannot be weakly cop-win.
 
The above observations suggest that dismantability is likely not the right notion to use with our definition of weakly cop-win graphs. Nevertheless, in the remainder of this section we show that some of the proof techniques from this paper can still be applied to dismantable graphs. In this respect, locally finite graphs seem to be particularly well behaved.

Let $G$ be a dismantable graph with dismantling order $<$ and domination map $\delta$. If possible we define a map $\rho_\nu\colon G \to G_{\geq \nu}$ by
\[
\rho_\nu(\mu) = \delta^{k(\nu,\mu)}(\mu)
\]
where
\[
k(\nu,\mu) = \min \{k \in \mathbb N \mid \delta^k(\mu) \geq \nu\}.
\]
Note that $\rho_\nu$ may not be defined for an arbitrary well order since $k(\nu,\mu)$ may well be infinite. Clearly whether or not $\rho_\nu$ is well defined does not only depend on the order isomorphism class of the dismantling order but also on the domination map. However, if the domination order is natural (that is, order isomorphic to a subset of $\mathbb N$) then the definition is always possible. An analogous argument as in Lemma \ref{lem:naturalorder} shows that for locally finite graphs we can always find such an order. From now on we will restrict ourselves to natural dismantling orders although some of the results may be true in a more general context.

Just like in the constructible case we get that $\rho_\nu$ (if it is defined) is a retraction with some additional properties. We only prove this fact for graphs with a natural dismantling order.

\begin{lem}
\label{lem:retraction-dismantable}
Let $G$ be a graph admitting a natural dismantling order. Then:
\begin{itemize}
\item The map $\rho_\nu$ it is a retraction. 
\item If $\mu$ and $\eta$ are connected by an edge, then also $\rho_{\nu+1}(\mu)$ and $\rho_\nu(\eta)$ are connected by an edge.
\end{itemize}
\end{lem}

\begin{proof}
For the first part define the map $\delta_\nu \colon G_{\geq \nu} \to G_{\geq\nu+1}$ by $\delta_\nu (\nu) = \delta(\nu)$ and $\delta_\nu (\mu) = \mu$ for $\mu > \nu$. Clearly $\delta_\nu$ is a retraction. Hence $\rho_\nu$ is a retraction because compositions of retractions are again retractions.

The second part is proved by induction. It is easy to check that the statement is true for $\nu = 1$. It is also clear that $\rho_{\nu}(\mu)$ is connected to $\rho_\nu(\eta)$ because $\rho_\nu$ is a retraction.  

In the induction step we distinguish two cases. If $\rho_{\nu+1}(\mu) = \rho_{\nu}(\mu)$ then it is connected to $\rho_\nu(\eta)$. Otherwise $\rho_{\nu+1}(\mu)$ dominates $\nu =\rho_{\nu}(\mu)$ in $G_{\geq \nu}$. In particular $\rho_{\nu+1}(\mu)$ is connected by an edge to every neighbour of $\rho_{\nu}(\mu)$ and hence also to $\rho_\nu(\eta)$.
\end{proof}

\begin{thm}
\label{thm:dismantable-copwin}
Let $G$ be a dismantable graph with a natural dismantling order $<$. Assume that $\rho_\nu$ is defined for every $\nu \in V$ and that for each pair $\nu,\mu$ of vertices there is some $\gamma$ such that $\rho_\gamma(\nu) = \rho_\gamma(\mu)$. Then $G$ is weakly cop-win.
\end{thm}

\begin{proof}
Informally the cop's strategy can be summarised as follows: First she tries to get to a vertex $c = \rho_c(r)$ where $r$ is the robber's position. Once she has achieved that, she follows a similar strategy as in Theorem \ref{thm:weaklycopwin} to catch the robber.

More formally let $k(c,r) = \min\{k \in \mathbb N \mid c\delta^k(r) \in E\}$ with the convention that the minimum over the empty set is $\infty$. The cop starts at $0$. For each consecutive step we define a strategy by
\[
	s(c,r) = 
	\begin{cases}
	\delta^{k(c,r)}(r) & \text{if } k(c,r)<\infty,\\
	\delta(c) & \text{if } k(c,r) = \infty. 
	\end{cases}
\]
It remains to show that this strategy is weakly winning. 

First we claim that if $k(c,r)$ is finite at some point, then the robber will be caught after finitely many rounds. To prove this claim observe that if $k(c,r)$ is finite, then after the cops move $c = \rho_\gamma(r)$ for some suitable $\gamma \in \mathbb N$, where $c$ and $r$ are the positions of the cop and the robber respectively. By the second part of Lemma \ref{lem:retraction-dismantable} we have $c = \rho_{\gamma'}(r)$ for some $\gamma'<\gamma$ after the next round. Continuing inductively we get $c = \rho_{0}(r) = r$ after at most $\gamma$ rounds.

Thus we only need to show that $k(c,r)$ cannot remain infinite forever if the robber visits one vertex infinitely often. So assume that there was a vertex $\nu$ which the robber visits infinitely often. By assumption there is some finite $\gamma$ such that $\rho_\gamma(\nu) = \rho_\gamma(0)$. Clearly, if $k(c,r) = \infty$ then the cop will arrive at $\rho_\gamma(0)$ after at most $\gamma$ steps. The next time the robber visits $\nu$ the cop will be at $\delta^n(\rho_\gamma(0)) = \delta ^m(\nu)$ for some $m,n \in \mathbb N$ which implies that $k(c,r)<\infty$. This completes the proof of the theorem.
\end{proof}

We would like to remark that a more careful analysis in the above proof shows, that graphs satisfying the conditions of Theorem \ref{thm:dismantable-copwin} are not only weakly cop-win, but even $C$-weakly cop-win. However, since the requirements are so much stronger than in Theorem \ref{thm:weaklycopwin} it is still more than likely that constructibility is more useful than dismantability in the study of weakly cop-win graphs.

\section{Outlook and open questions}

While the present paper answers some of the questions asked in \cite{zbMATH01533339}, the results raise several new questions about weakly cop-win graphs. The first and most obvious question is of course, whether Theorem \ref{thm:nowakowski} carries over to the infinite case with our notion of weakly cop-win graphs. Since one implication is already covered by Theorem \ref{thm:weaklycopwin} it only remains to answer the following question.

\begin{question}
Is every weakly cop-win graph constructible?
\end{question}

A first step towards the answer of this question could consist of treating locally finite graphs. One way to attack this question could be using Theorem \ref{thm:constructible} which in the locally finite case characterises the constructible graphs as the graphs for which there is a protective strategy. 

Of course, one could also hope for a similar characterisation in the general case as a next step towards the above question.

\begin{question}
Is there a similar characterisation to Theorem \ref{thm:constructible} for arbitrary (non-locally finite) graphs?
\end{question}

One may also ask, under what circumstances a cop win strategy can be adapted into a (time dependent) protective strategy. In particular, a positive answer to the following question would immediately imply that the constructible graphs are exactly the weakly cop-win graphs.

\begin{question}
Does every weakly cop-win graph admit a (time dependent) protective strategy?
\end{question}

\section*{Acknowledgements}

The author would like to thank Przemys\l{}av Gordinowicz for pointing out the example in \cite{zbMATH01533339} that shows that there is a dismantable graph which is not weakly cop-win and Richard Johnson for helping to improve the manuscript.

\bibliographystyle{abbrv}
\bibliography{sources}

\end{document}